\documentclass[11pt]{amsart}
\usepackage{amsfonts,url,color}

\newtheorem{theorem}{Theorem}[section]
\newtheorem{prop}[theorem]{Proposition}
\newtheorem{lemma}[theorem]{Lemma}
\newtheorem{cor}[theorem]{Corollary}

\newtheorem{preprob}[theorem]{Problem}
\newenvironment{prob}{\preprob\rm}{\endpreprob}
\newtheorem{preconj}[theorem]{Conjecture}
\newenvironment{conj}{\preconj \rm}{\endpreconj}
\newtheorem{preqn}[theorem]{Question}
\newenvironment{qn}{\preqn\rm}{\endpreqn}
\newtheorem{predefn}[theorem]{Definition}
\newenvironment{defn}{\predefn\rm}{\endpredefn}
\newtheorem{preex}[theorem]{Example}
\newenvironment{ex}{\preex\rm}{\endpreex}

\newcommand{\Aut}{\mathop{\mathrm{Aut}}\nolimits}
\newcommand{\out}{\mathop{\mathrm{Out}}\nolimits}
\newcommand{\Soc}{\mathop{\mathrm{Soc}}\nolimits}
\newcommand{\sym}{\mathrm{S}}
\newcommand{\alt}{\mathrm{A}}
\newcommand{\psl}{\mathop{\mathrm{PSL}}\nolimits}
\newcommand{\psu}{\mathop{\mathrm{PSU}}\nolimits}
\newcommand{\pgl}{\mathop{\mathrm{PGL}}\nolimits}
\newcommand{\agl}{\mathop{\mathrm{AGL}}\nolimits}
\newcommand{\gl}{\mathop{\mathrm{GL}}\nolimits}
\newcommand{\slin}{\mathop{\mathrm{SL}}\nolimits}

\newcommand{\gaml}{\mathop{\Gamma\mathrm{L}}\nolimits}
\newcommand{\rem}{\mathrm{m}}
\newcommand{\rg}{\Gamma}
\newcommand{\rc}{\mathrm{c}}
\newcommand{\rw}{\mathrm{w}}
\newcommand{\Core}{\mathop{\mathrm{Core}}\nolimits} 
\newcommand{\frat}{\mathop{\mathrm{Frat}}\nolimits}
\newcommand{\End}{\mathop{\mathrm{End}}\nolimits}
\newcommand{\fit}{\mathop{\mathrm{Fit}}\nolimits}
\newcommand{\perm}{\mathop{\mathrm{Sym}}}
\newcommand{\Fix}{\mathop{\mathrm{Fix}}\nolimits}
\newcommand{\Magma}{{\sc MAGMA}}

\begin{document}

\title{Generating sets of finite groups}
\author{Peter J. Cameron} \address{Peter J. Cameron, 
University of St Andrews, Mathematical Institute, St Andrews, Fife KY16 9SS, Scotland}\email{pjc20@st-andrews.ac.uk}
\author{Andrea Lucchini} \address{Andrea Lucchini, Dipartimento di Matematica,
Universit\`a degli studi di Padova, Via Trieste 63, 35121
Padova, Italy} \email{lucchini@math.unipd.it}
\author{Colva M. Roney-Dougal}\address{Colva M. Roney-Dougal, 
University of St Andrews,
  Mathematical Institute, St Andrews, Fife KY16 9SS, Scotland} 
\email{colva.roney-dougal@st-andrews.ac.uk}

\thanks{The second
 and third authors were supported by Universit\`a di Padova (Progetto
  di Ricerca di Ateneo: Invariable generation of groups).}
\date{\today}
\maketitle

\begin{abstract}
We investigate the extent to which the exchange relation holds in
finite groups $G$. We define a new equivalence relation
$\equiv_{\rem}$, 
where two elements
are equivalent if each can be substituted for the other in any
generating set for $G$. We then refine this to 
a new sequence $\equiv_\rem^{(r)}$ of equivalence relations by saying that 
$x \equiv_\rem^{(r)}y$  if each can be substituted for the other in any
$r$-element generating set.
The relations $\equiv_\rem^{(r)}$ become finer as $r$ increases, and
we define a new group invariant
$\psi(G)$ to be the value of $r$ at which they stabilise to $\equiv_{\rem}$. 

Remarkably, we are able to prove that if
$G$ is soluble then $\psi(G) \in \{d(G), d(G) +1\}$, where $d(G)$ is
the minimum number of generators of $G$, and  to classify the finite soluble
groups $G$ for which $\psi(G) = d(G)$. For insoluble $G$, we show that
$d(G) \leq \psi(G) \leq d(G) + 5$. However, we know of no examples of
groups $G$ for which $\psi(G) > d(G) + 1$. 

As an application, we look at the \emph{generating graph} of $G$,
whose vertices are the elements of $G$, the edges being the $2$-element
generating sets. 
Our relation $\equiv_{\rem}^{(2)}$ enables 
us to calculate $\Aut(\Gamma(G))$ for all soluble groups $G$ of
nonzero spread, and give detailed structural information about
$\Aut(\Gamma(G))$ in the insoluble case.  
\end{abstract}

\section{Introduction}

It is well known that generating sets for groups 
are far more complicated
than generating sets for, say, vector spaces. The latter satisfy the
exchange axiom, and hence any two irredundant 
sets
have the same cardinality. According to
the Burnside Basis Theorem, a similar property holds for groups of
prime power order.

Our starting point is the observation that, in order to
understand better the generating sets for arbitrary finite groups, we should investigate
the extent to which the exchange property holds. We define an equivalence
relation $\equiv_{\rem}$ 
on a finite group $G$, in which two elements are equivalent if each can be
substituted for the other in any generating set for $G$. Then two elements
are equivalent if and only if they lie in the same maximal subgroups
of $G$.

We refine this relation to a sequence of relations
$\equiv_{\rem}^{(r)}$ whose terms depend on a positive integer
$r$, where two elements are equivalent if each can be substituted
for the other in any $r$-element generating set. 
The relations $\equiv_\rem^{(r)}$ become finer as $r$ increases; we
observe in Lemma~\ref{lem:m-basics}  that the smallest
value of $r$ for which $\equiv_\rem^{(r)}$ is not the universal relation is
the minimum number $d(G)$ of generators of $G$. 

We define
a new group invariant
$\psi(G)$ to be the value of $r$ at which the relations $\equiv_{\rem}^{(r)}$
stabilise to $\equiv_{\rem}$. Remarkably, it turns out 
(see Corollary~\ref{cor:sol}) that if
$G$ is soluble then $\psi(G) \in \{d(G), d(G) +1\}$. In 
Theorem~\ref{thm:small_d} we even succeed in giving a precise structural description of
the finite soluble groups $G$ for which $\psi(G) = d(G)$.

In the general case, we show in 
Corollary~\ref{cor:insol} and Proposition~\ref{prop:mu}
that $\psi(G) \leq d(G) + 5$, with tighter
bounds when $G$ is (almost) simple. However, we know of no examples of
groups $G$ for which $\psi(G) >  d(G)+1$. 

The relation $\equiv_\rem$ can be a little
tricky to work with, so in Section~\ref{sec:c} we introduce a far
simpler relation, by defining
$x \equiv_{\rc} y$ if $\langle x \rangle = \langle y \rangle$. 
This is clearly a refinement of $\equiv_{\rem}$, and
 provides an easy-to-calculate upper bound on the number of
$\equiv_{\rem}$-classes, and lower bound on their sizes. In
  Theorem~\ref{thm:sol_c} we characterise the soluble groups $G$ on
  which these two relations coincide; it would be very interesting to
  determine for which insoluble groups they are equal.

As an  application, we notice that 
the relation $\equiv_\rem^{(2)}$ is particularly interesting for two-generator
groups. Such groups $G$ 
have long been studied by means of the \emph{generating graph},
whose vertices are the elements of $G$, the edges being the $2$-element
generating sets. 
 The generating graph was defined by Liebeck and
Shalev in \cite{LiebeckShalev96}, and has been further investigated by
many authors: see for example 
\cite{beghm, bglmn, CL3, gk, lmclique, lm, LMRD, atsim} for some
of the range of questions that have been considered. Many deep
structural results about finite groups can be expressed in terms of
the generating graph. 

We notice that two group elements are $\equiv_\rem^{(2)}$-equivalent if
and only if they have the same neighbours in the generating graph. By
identifying the vertices in each equivalence class, we obtain a reduced
graph $\overline{\Gamma}(G)$, which has many fewer vertices, but the same
spread, clique number and chromatic number, amongst other properties. 
We conjecture 
 that in a group $G$ of nonzero spread,
 the equivalence relations
$\equiv_\rem$ and $\equiv_\rem^{(2)}$ coincide.

The automorphism groups of generating graphs are extremely large, and
their study has up to now seemed intractable. However, 
we show in Theorem~\ref{thm:aut_g} that the automorphism group of
${\Gamma}(G)$ has a very compact description in terms of the sizes of
the $\equiv_{\rem}^{(2)}$-classes of $G$, and the group
$\Aut(\overline{\Gamma}(G))$. Using this, we are able to give a precise
description of 
the automorphism groups of the generating graphs of all
soluble groups of nonzero spread, and a detailed description in
the insoluble case.

We have carried out many computational experiments on small 
insoluble groups $G$ of nonzero spread. In each case we found that
$\psi(G) = 2$, and that $\Aut(\Gamma(G))$ is completely and
straightforwardly determined by
the sizes of the $\equiv_{\rem}^{(2)}$-classes and
$\Aut(G)$. 

The paper is structured as follows. In Section~\ref{sec:equivs} we
study the relations $\equiv_\rem$ and $\equiv_{\rem}^{(r)}$, and the related invariant
$\psi(G)$. 
In Section~\ref{sec:c}
we look at the relation $\equiv_{\rc}$. In Section~\ref{sec:gen} we
introduce the generating graph $\Gamma(G)$ and the reduced generating
graph $\overline{\Gamma}(G)$, and then in Section~\ref{sec:aut} we
study the group $\Aut(\Gamma(G))$ for groups $G$ of nonzero spread. 

\section{A hierarchy of equivalences}\label{sec:equivs}

\subsection{Definitions and elementary results}

We shall now introduce our main families of relations, and establish a
few basic results concerning them.

\begin{defn}
Let $G$ be a finite group. 
We define an equivalence relation $\equiv_\rem$ ($\rem$ for ``maximal
subgroups'') on $G$ by letting  $x\equiv_\rem y$ if and only if
$x$ and $y$ lie in exactly the same maximal subgroups of $G$. 
\end{defn}

Note that the $\equiv_\rem$-class containing the identity is precisely the
Frattini subgroup  of $G$,  and any $\equiv_\rem$-class is a union of
cosets of the Frattini subgroup.

The equivalence relation $\equiv_\rem$ can also be characterised by a
substitution property:

\begin{prop}
Let $G$ be a finite group, and let $x$ and $y$ be elements of $G$.
Then $x\equiv_\rem y$ if and only if
\[(\forall r)(\forall z_1,\ldots,z_r\in G)((\langle x,z_1,\ldots,z_r\rangle=G)
\Leftrightarrow(\langle y,z_1,\ldots,z_r\rangle=G)).\]
\end{prop}

\begin{proof}
Suppose first that $\langle x,z_1,\ldots,z_r\rangle=G$ but
$\langle y,z_1,\ldots,z_r\rangle\ne G$. Then there is a maximal subgroup $M$
of $G$ containing $y,z_1,\ldots,z_r$. Clearly $x\notin M$; so
$x\not\equiv_\rem y$.

Conversely, suppose that $x\not\equiv_\rem y$, so that (without loss of
generality) there is a maximal subgroup $M$ containing $y$ but not $x$. Choose
generators $z_1,\ldots,z_r$ for $M$. Then $\langle y,z_1,\ldots,z_r\rangle=M$,
but $\langle x,z_1,\ldots,z_r\rangle$ properly contains $M$, and so is equal
to $G$.
\end{proof}

This means that, when considering generating sets (of any cardinality) for a
group $G$, we may restrict our attention to subsets of a set of 
$\equiv_\rem$-class representatives.

\begin{defn}
For any positive integer $r$, define equivalence relations $\equiv_\rem^{(r)}$
by the rule that $x\equiv_\rem^{(r)}y$ if and only if
\[(\forall z_1,\ldots,z_{r-1}\in G)((\langle x,z_1,\ldots,z_{r-1}\rangle=G)
\Leftrightarrow(\langle y,z_1,\ldots,z_{r-1}\rangle = G)).\]
\end{defn}


\begin{lemma}\label{lem:m-basics}
\begin{enumerate}
\item[(1)] The relations $\equiv_\rem^{(r)}$ get finer as $r$ increases.
\item[(2)] The smallest value of $r$ for which $\equiv_\rem^{(r)}$
	is not the universal relation is $d(G)$. 
        For $r = d(G)$, there are at least $r+1$
        equivalence classes. 
\item[(3)] The limit value of this sequence of relations is
$\equiv_\rem$.
\end{enumerate}
\end{lemma}

\begin{proof}
(1) Choosing
$z_{r-1}$ to be the identity we see that $x\equiv_\rem^{(r)}y$ implies
$x\equiv_\rem^{(r-1)}y$. 

(2) The first claim is clear, for this second, notice that the
 identity and the elements of any $d(G)$-element
generating set are pairwise inequivalent.

(3) This is clear.
\end{proof}

\begin{defn}
Let $\psi(G)$ be the value of $r$ for which the equivalences
$\equiv_\rem^{(r)}$ stabilise, that is, the least $r$ such that
$\equiv_\rem^{(r)}$ coincides with the limiting relation $\equiv_\rem$.
\end{defn}

\subsection{Bounds on $\psi(G)$}

In this subsection, we 
prove various upper and lower
bounds on $\psi(G)$ in terms of other numerical invariants of $G$. 
We start with some straightforward lower bounds on $\psi(G)$.
\begin{lemma}\label{lem:lower}
Let $G$ be a finite group, and let $d = d(G)$. Then $\psi(G) \geq d$, and
if $G$ has a normal subgroup $N$ such that $N \not\leq \frat(G)$ and
$d(G/N) = d$, then $\psi(G)
\geq d + 1$.
\end{lemma}

\begin{proof}
The first claim is immediate from Lemma~\ref{lem:m-basics}(2). 
For the second, notice that
elements of $N$ lie in no $d$-element generating set of $G$, and so
are $\equiv_\rem^{(d)}$-equivalent to the identity. However, the
$\equiv_\rem$-equivalence class of the identity is $\frat(G)$. 
\end{proof}

These lower bounds are best possible in a very strong sense: 
we know of no groups that do
not attain them.

\begin{prob}
Is it true that if $G$ is a finite group, then $\psi(G)  \in \{d(G),
d(G) + 1\}$?
\end{prob}

Whilst we are not able to answer this question in general, in the rest
of this subsection we prove some upper bounds on $\psi(G)$. In
particular, in Corollary~\ref{cor:sol} we show that if $G$ is soluble then
$\psi(G) \leq d(G) + 1$.

\begin{defn}Let $G$ be a finite group and 
let $M$ be a core-free maximal subgroup of $G$.
For every $g\in G\setminus M$,
let $\delta_{G,M}(g)$ be
the smallest cardinality of  a subset $X$ of $M$ with the property that
$G=\langle g,X\rangle$ and let
$$\nu_M(G)=\sup_{g\notin M}\delta_{G,M}(g).$$
\end{defn}
Notice that $\nu_M(G) \leq d(M)$.

\begin{defn}
Let $\tilde m(G)$ be the maximum of $\nu_{M/N}
(G/N)$  
over 
all maximal subgroups
$M$ of $G$, where $N = \Core_G(M)$.
\end{defn} 

\begin{theorem}\label{main}
$\psi(G)\leq \max\{\tilde m(G),d(G)\}+1.$
\end{theorem}

Before proving this 
result, we briefly recall a necessary definition and result. Given a subset
$X$ of 
a finite group $G,$ we will denote by $d_X(G)$ the smallest cardinality
of a set of elements of $G$ generating $G$ together with the elements 
of $X.$ The following 
generalizes
a result originally
obtained by W. Gasch\"utz \cite{Ga} for $X=\emptyset.$ 

\begin{lemma}[\cite{CL3} Lemma 6] \label{modg} Let $X$ be a subset of $G$ and $N$ a normal subgroup of $G$ and suppose that
	$\langle g_1,\dots,g_k, X\rangle N=G.$ 
	If $k\geq d_X(G),$ then there exist
$n_1,\dots,n_k\in N$ so that $\langle g_1n_1,\dots,g_kn_k,X\rangle=G.$
\end{lemma}

\noindent \textit{Proof of Theorem \ref{main}.}
Let  $t=\max\{\tilde m(G),d(G)\}.$ Since the relations
$\equiv_\rem^{(r)}$ become finer with $r$, it suffices to prove that
if 
$x$ and $y$ are two elements of $G$ and $x\not\equiv_\rem y,$ then 
$x \not\equiv_\rem^{(t+1)} y.$ So assume that
$x\not\equiv_\rem y.$ It is not restrictive to assume that 
there exists a maximal subgroup $M$ of $G$ such that 
$x\notin M$ and
$y\in M.$ Let $N=\Core_G(M)$ and let  $X=\{x \}.$ Since $t\geq \tilde m(G),$ we have $t\geq \nu_{M/N}(G/N),$ hence there exist $g_1,\dots,g_t\in M$ such that 
$\langle x, g_1,\dots,g_t\rangle N = G.$ Moreover
$t\geq d(G)\geq d_X(G).$ So we deduce from Lemma \ref{modg} that there
exist $n_1,\dots,n_t\in N$ such that $G=\langle
x,g_1n_1,\dots,g_tn_t\rangle.$ On the other hand $\langle
y,g_1n_1,\dots,g_tn_t\rangle \leq M.$  Hence $x
\not\equiv_\rem^{(t+1)} y.$ 
\hfill $\Box$

\medskip

We are now able to prove a tight upper bound on $\psi(G)$ for all
finite soluble groups $G$. 

\begin{cor}\label{cor:sol}
If $G$ is a finite soluble group, then $\psi(G)\leq d(G)+1.$ 
\end{cor}
\begin{proof}
Let $M$ be a maximal subgroup of $G$, and let $K = \Core_G(M)$. 
Then $\tilde G  = G/K$ is a  soluble group with a faithful primitive
action on the
cosets of $M/K$, and $d(\tilde G)\leq d(G).$ Moreover
$M /K$ is a complement in $\tilde G$ of $\Soc (\tilde G),$ so 
$\nu_{M/K} (G /K)\leq d(M /K)=d(\tilde G/\Soc(\tilde G))\leq d(\tilde
G)\leq d(G).$
This holds for every maximal subgroup of $G$, so  $\tilde m(G)\leq d(G)$
and the conclusion follows from Theorem \ref{main}.
\end{proof}

Now we prove an upper bound on $\psi(G)$ for an arbitrary finite group $G$. 

\begin{cor}\label{cor:insol}
If $G$ is a finite group, then $\psi(G)\leq d(G)+5$. Furthermore, if
$G$ is simple, then $\psi(G) \leq 5$, and  if $G$ is almost simple
then $\psi(G) \leq 7$. 
\end{cor}

\begin{proof}
Burness, Liebeck and Shalev prove  (see \cite[Theorem 7]{bls}) 
that the point stabiliser of a $d$-generated finite primitive 
permutation group can be generated by $d+4$ elements. 
Hence if $G$ is a finite group, then
$\tilde m(G) \leq d(G)+4$ and our first claim 
follows from Theorem \ref{main}. 

In the same paper (see \cite[Theorems
1 and 2]{bls}) they show that any maximal subgroup of a finite simple
group can be generated by $4$ elements, and that any maximal subgroup
of an almost simple group can be generated by $6$ elements. Hence our
final two claims follow in the same way. 
\end{proof}



We conclude this subsection by
mentioning a relationship with another well-known parameter,
$\mu(G)$, the maximum size of a minimal generating set for $G$ (a
generating set for which no proper subset generates), studied by Diaconis
and Saloff-Coste, Whiston, Saxl, and others~\cite{ds,Jambor,whiston}.

\begin{prop}\label{prop:mu}
Let $G$ be a finite group. Then 
$\psi(G)\le \mu(G)$. Hence if $G = \psl_{2}(p)$ with $p \not\in \{7,
11, 19, 31\}$ then $\psi(G) \leq 3$, and $\psi(\psl_{2}(p)) \leq 4$ in
the remaining cases.
\label{p:mud}
\end{prop}

\begin{proof} 
To prove that $\psi(G)\le\mu(G)$,
we show that if $\mu=\mu(G)$, and
$x\equiv_\rem^{(\mu)}y$, then $x\equiv_\rem y$. So suppose that 
$x\equiv_\rem^{(\mu)}y$, and let $G=\langle x,z_1,\ldots,z_{r-1}\rangle$.

\noindent \textbf{Case $r\le\mu$.} Since the relations $\equiv_\rem^{(r)}$ get
finer as $r$ increases, in this
case $G=\langle y,z_1,\ldots,z_{r-1}\rangle$.

\noindent \textbf{Case $r>\mu$.} In this case, our generating set is larger than
$\mu$, and so some element is redundant.
If $x$ is redundant, then $G=\langle z_1,\ldots,z_{r-1}\rangle
=\langle y,z_1,\ldots,z_{r-1}\rangle$, as required.
Suppose that $x$ is not redundant. Then $G$ is generated by a subset of the
given generators of size $\mu$ including $x$, without loss of generality
$\{x,z_1,\ldots,z_{\mu-1}$\}. Since, by assumption, $x\equiv_\rem^{(\mu)}y$, we
have $G=\langle y,z_1,\ldots,z_{\mu-1}\rangle
=\langle y,z_1,\ldots,z_{r-1}\rangle$.

The final claim follows from \cite{Jambor}, where the stated bounds on
$\mu(\psl_{2}(p))$ are determined.
\end{proof}

In general $\mu(G)$ can be much larger than $d(G)$. For example,
if $G$ is soluble, than $m(G)-d(G)\geq \pi(G)-2$ (see \cite[Corollary 3]{lucchini13})
and in any case $\mu(G)$ is at least the number of complemented factors in a
chief series of $G$ (see \cite[Theorem 1]{lucchini13}). 
Hence the difference $\mu(G)-d(G)$
(and consequently, by Corollary 1.10, the difference $\mu(G)-\psi(G)$)
can be arbitrarily large. 

\subsection{Groups with $\psi(G) = d(G)$}

In this subsection, we study groups $G$ for which $\psi(G) = d(G)$;
in particular in Theorem~\ref{thm:small_d} we describe the structure
of such soluble groups $G$.

\begin{defn}\label{def:efficient}
A finite group $G$ is \emph{efficiently generated} if for all $x\in G$, 
$d_{\{x\}}(G)=d(G)$ implies that $x\in \frat(G).$
\end{defn}

\begin{lemma}\label{lemma1}
If $\psi(G)=d(G),$ then
$G$ is efficiently generated.
\end{lemma}
\begin{proof} Let $d=d(G).$
If $G$ is not efficiently generated, then there exists $x\notin
\frat(G)$ such that $d_{\{x\}}(G)=d.$ 
This implies in particular $x \equiv_\rem^{(d)} 1.$ However
since $x\notin \frat(G),$ we have $x \not\equiv_\rem 1$, hence $\psi(G)>d.$
\end{proof}

\begin{lemma}\label{lemma2}
If $G$ is efficiently generated and $\tilde m(G)< d(G),$ then $\psi(G)=d(G).$
\end{lemma}
\begin{proof} Let $d=d(G)$. 
By Theorem \ref{main}, our assumption that $\tilde m(G) < d(G)$
implies that $\psi(G)\leq d+1$, and hence that $\equiv_\rem^{(d+1)}$ 
coincides with $\equiv_\rem$. It therefore
suffices to prove that if $x \not \equiv_\rem^{(d+1)} y,$ then 
$x \not \equiv_\rem^{(d)} y.$ 

Assume that $x \not \equiv_\rem^{(d+1)} y$ 
and let $d_x=d_{\{x\}}(G)$ and $d_y=d_{\{y\}}(G).$ It is clear that $d_x, d_y \geq
d-1$.
If $d_x = d_y = d$,  then our assumption that $G$ is efficiently
generated implies that
$x,y \in \frat(G)$,  and hence that $x \equiv_\rem y$, a contradiction. 
Therefore we may assume that $d_x=d-1;$ in particular 
$G=\langle x,g_1,\dots,g_{d-1}\rangle$ for some 
$g_1,\dots,g_{d-1}\in G.$ If $d_y=d,$ then
$G\neq \langle y,g_1,\dots,g_{d-1}\rangle$ and therefore $x
\not \equiv_\rem^{(d)} y$, and we are done.

	So assume that $d_x=d_y=d-1$. Since $x \not\equiv_\rem y$,
        without loss of generality there 
exists a maximal subgroup $M$ of $G$ such that $x\notin M,$
$y\in M.$ Let $N=\Core_G(M).$ Since $d-1\geq \tilde m(G),$  
there exist $g_1,\dots,g_{d-1}\in M$ such that
	$\langle x, g_1,\dots,g_{d-1}\rangle N = G.$ As $d_x=d-1,$  we deduce from Lemma \ref{modg} that there exist $n_1,\dots,n_{d-1}\in N$ such that $G=\langle x,g_1n_1,\dots,g_{d-1}n_{d-1}\rangle.$ On the other hand $\langle y,g_1n_1,\dots,g_{d-1}n_{d-1}\rangle \leq M.$  Hence $x \not\equiv_\rem^{(d)} y.$
\end{proof}

Notice that if $d(M)<d(G)$ for every maximal subgroup $M$ of $G$, then $G$ is efficiently generated. Indeed if $x\notin \frat(G),$ then there exists a maximal subgroup $M$ of $G$ with $x\notin M$ and consequently $d_{\{x\}}(G)\leq d(M)<d(G).$
But then from Lemma~\ref{lemma2} 
we deduce the following result.

\begin{cor}  If $d(M)<d(G)$ for every maximal subgroup $M$ of $G$, then $\psi(G)=d(G).$ \end{cor}

\begin{lemma}\label{lem:soleffic}
Let $G$ be a finite soluble group. If $G$ is efficiently generated then $\tilde m(G)<d(G).$
\end{lemma}

\begin{proof} If suffices to prove that for every maximal subgroup $M$ of $G,$ we have
	$d(M/\Core_G(M))<d(G)=d.$
Assume otherwise. Then there exists a maximal subgroup $M$ of $G$ such
that $d(M/N)=d$ (where $N=\Core_G(M)$). Furthermore, $G/N=A/N : M/N$
and 
$\frat(G)\leq N.$ Let $a\in A \setminus \frat(G)$. Then $d_{\{a\}}(G)=d,$
contradicting the assumption that $G$ is efficiently generated. 
\end{proof}

The following result is now immediate from Lemmas~\ref{lemma1}
 and \ref{lem:soleffic}.

\begin{cor}\label{psisol}
Let $G$ be a finite soluble group. 
Then $\psi(G)  = d(G)$ if and only if $G$ is efficiently generated.
\end{cor}

\begin{theorem}\label{thm:small_d}
A finite soluble group $G$  satisfies $\psi(G) = d(G)$ 
if and 
only if either $G$ is a finite $p$-group or there
exist a finite vector space $V$, a nontrivial irreducible soluble subgroup $H$ of $\Aut(V)$
and an integer $d>d(H)$ such that
$$G/\frat(G)\cong V^{r(d-2)+1} : H,$$
where $r$ is the dimension of $V$ over $\End_H(V)$ and $H$ acts in the same way on
each of the $r(d-2)+1$ factors.
\end{theorem}

\begin{proof}
Assume that $G$ is soluble group with $\psi(G) = d(G) = d$ 
and let $F=\frat(G)$. By Corollary~\ref{psisol}, $G$ is efficiently
generated.  If $N$ is a normal subgroup of $G$ 
properly containing $F,$ then
$d(G/N)<d$ (otherwise we would have $d_{\{n\}}(G)=d$ for every $n \in
N$). 
So $G/F$ has the property that every proper quotient can be generated by $d-1$ elements, but $G/F$ cannot. The groups with this property have been studied in \cite{dvl}. By \cite[Theorem 1.4 and Theorem 2.7]{dvl} either $G/F$ is an elementary abelian $p$-group of rank $d$ (and consequently $G$ is a finite $p$-group) or there exist a finite vector space $V$ and a nontrivial irreducible soluble subgroup $H$ of $\Aut(V)$
such that $d(H)<d$ and $G/\frat(G)\cong V^{r(d-2)+1} :  H,$
where $r$ is the dimension of $V$ over $\End_H(V).$ 

Conversely, if $G$ is a finite $p$-group it follows 
immediately from Burnside's basis theorem that $G$ is 
efficiently generated, and so $\psi(G) = d(G)$ by Corollary~\ref{psisol}.
Clearly a group $G$ is efficiently generated if and only if 
$G/\frat(G)$ is efficiently generated. So to conclude the 
proof it suffices to prove that if $H$ is a $(d-1)$-generated 
soluble irreducible subgroup of $\Aut(V)$ and $r$ is the 
dimension of $V$ over $F=\End_H(V),$ then
$X= V^{r(d-2)+1} :  H$ is efficiently generated. Notice that 
$d(X)=d$, so we have to prove that $d_{\{x\}}(X)\leq d-1$ for 
every $x\neq 1$. Let $n=r(d-2)+1.$ Fix a nontrivial element 
$x=(v_1,\dots,v_n)h\in X$
and let $a=\dim_F C_V(h)$ and $b=n-\dim_F \langle [V,h],
v_1,\dots,v_n\rangle+\dim_F [V,h].$
By \cite[Lemma 5]{bias} we have $d_{\{x\}}(X)\leq d-1$ if and 
only if $a+b-1<r(d-1).$
 If $h\neq 1,$ then $a\leq r-1$ and $b\leq n;$ if $h=1,$ 
then $a\leq r$ and $b\leq n-1.$ In any case
 $a+b-1\leq r+n-2 = r+r(d-2)-1 < r(d-1).$
\end{proof}

Apart from $p$-groups, there are many examples of soluble groups that
are 
efficiently generated. The smallest example of a soluble group which is not
efficiently generated is $\sym_4$ (we have $d_{\{x\}}(\sym_4)=2$ for every
$x$ in the Klein subgroup): by the previous results we can conclude
that $\psi(\sym_4)=3$.

\begin{prob} 
Characterise the insoluble groups that are efficiently generated.
\end{prob}

\subsection{Calculating $\equiv_\rem$}\label{subsec:calculate}

Whilst we have not been able to determine $\psi(G)$ for an arbitrary
group $G$, we have calculated it for many small almost simple groups $G$
with $d(G) = 2$. It is computationally expensive to repeatedly
calculate whether various sets of elements generates a group. In
this subsection we describe an efficient way to calculate
$\equiv_\rem$- and $\equiv_\rem^{(2)}$-classes in a group, and present
a theorem summarising the results of these calculations.

The equivalence relation $\equiv_\rem$ can be thought of another way. Construct the permutation action
of $G$ which is the disjoint union of the actions on the cosets of maximal
subgroups, one for each conjugacy class. Let $\Omega$ be the domain of this
action. For brevity, we call this the \emph{m-universal action} of $G$.

\begin{lemma}\label{lem:m-universal}
Let $G$ be a finite group, and let $x, y \in G$ and $S \subseteq G$.
\begin{enumerate}
\item[(1)] $x\equiv_\rem y$ if and only if $x$ and $y$ have
the same fixed point sets in the m-universal action of $G$. 
\item[(2)] $G = \langle S \rangle$ if and only if the intersection of
  the fixed point sets of elements of $S$ 
in the m-universal action of $G$ is empty.
\end{enumerate}
\end{lemma}

\begin{proof}
Notice that in the orbit corresponding to
a non-normal maximal subgroup $M$, the point stabilisers are the conjugates
of $M$; whereas, if $M$ is normal, then its elements fix every point in the
corresponding orbit, while the elements outside $M$ fix none. Hence
the fixed point set of an element $x$ describes precisely which
maximal subgroups of $G$ contain $x$, and (1) follows. 
For (2), notice that $G = \langle S\rangle$ if and only if $S$ is
contained in no 
maximal subgroup of $G$.
\end{proof}

\begin{defn} A permutation group action 
has \emph{property $\mathcal{G}$} if it satisfies:
each set $S$ of group elements generates the group if and only if the
fixed-point sets  of elements of $S$ have empty
intersection.
\end{defn}


\begin{lemma}
The m-universal action is the smallest degree permutation action of $G$ with
property $\mathcal{G}$. 
\end{lemma}

\begin{proof}
First notice that by Lemma~\ref{lem:m-universal}(2),  the m-universal action has
property $\mathcal{G}$. 
Now suppose that we have an action of $G$ with property $\mathcal{G}$. We must
show that it contains the m-universal action. So let $M$ be a maximal subgroup
of $G$. Choose generators $g_1,\ldots,g_r$ of $M$. Since these elements do not
generate $G$, property $\mathcal{G}$ implies that they have a common fixed
point, say $\omega$. Thus $M\le G_\omega$, and maximality of $M$ implies
equality. So the coset space of $M$ is contained in the given action. Since
this holds for all maximal subgroups $M$, we are done.
\end{proof}

Our algorithm to test whether $\psi(G) = 2$ proceeds as
follows, on input a finite group $G$. 

\begin{enumerate}
\item[(1)] Construct the maximal subgroups of $G$, and hence the
  m-universal action of $G$.
\item[(2)] For each $g \in G$, compute the fixed point set $\Fix(g)$ of $g$ in the
  m-universal action, and hence construct a set of equivalence class
  representatives for the $\equiv_\rem$-classes of $G$. 
\item[(3)] For each pair $x, y$ of distinct $\equiv_\rem$-class
  representatives, check that there exists a $z \in G$ such that
  either $\Fix(x) \cap \Fix(z) = \emptyset$ and $\Fix(y) \cap \Fix(z)$
  is non-empty, or \emph{vice versa}.
\end{enumerate}

If the test in Step 3 succeeds for all distinct $x$ and $y$, then the
set of distinct $\equiv_\rem$-class representatives is also a set of
distinct $\equiv_\rem^{(2)}$-class representatives. That is, $\psi(G)
= 2$.

We have implemented the
algorithm in \Magma \, \cite{Magma}, and used it to prove the following:

\begin{theorem}\label{thm:small_as} 
Let $G$ be an almost simple group with socle of order less than $10000$
such that all proper quotients of $G$ are cyclic. Then $\psi(G) = 2$.
\end{theorem}

The socle of such a group $G$ is one of: $\alt_n$ for $5 \le n \le 7$,
$\psl_2(q)$ for $q \leq 27$ a prime power,
$\psl_3(3)$, $\psu_3(3)$ or the sporadic group $\mathrm{M}_{11}$.  

The only almost simple groups with socle of order less than 10000 with
a proper non-cyclic quotient are $\alt_{6}.2^2$ and $\psl_2(25).2^2$. 
Using similar
ideas to the above we were able to show that $\psi(\alt_{6}.2^2) = 3$. 

Notice that in all of these instances, the lower bounds from
Lemma~\ref{lem:lower} are attained. 

\section{$c$-equivalence}\label{sec:c}

In this section we define another equivalence relation, which can be
used to give an easy-to-calculate upper bound on the number of
$\equiv_{\rem}$-classes, and investigate when this new relation
coincides with $\equiv_{\rem}$.

\begin{defn}
Let $G$ be a finite group, and let $x, y \in G$. We define
$x\equiv_\rc y$ if $\langle x\rangle=\langle y\rangle$. We use $\rc$ for
\emph{cyclic}. 
\end{defn}

The following is clear.

\begin{lemma}
Let $G$ be a finite group.
For all $x, y \in G$, if $x \equiv_{\rc} y$ then $x \equiv_{\rem}
y$. Hence if $n$ is the order of an element of $G$, then at least one
$\equiv_{\rem}$-class of $G$ contains at least $\phi(n)$ elements.
\end{lemma}

The converse implication of the first statement holds for
many groups (including $\sym_n$ and $\alt_n$ for $n \in \{5, 6\}$,
and $\psl_2(q)$ for $q \in \{7,11,13\}$), but not for all groups. 

\begin{prop}\label{metab} 
Let $G$ be a finite group. If
the relations $\equiv_\rem$ and $\equiv_\rc$ coincide, then
\begin{enumerate}
	\item[(1)] $\frat(G)=1;$
	\item[(2)] if $G$ is soluble then every minimal normal subgroup of $G$ is cyclic;
	\item[(3)] if $G$ is soluble then $G$ is  metabelian.
\end{enumerate}
\end{prop}
\begin{proof}
(1) All of the elements of $\frat(G)$ are $\equiv_\rem$-equivalent. 

\noindent (2) Let $G$ be soluble and 
let $N$ be a minimal normal subgroup of $G.$ Every maximal
subgroup 
of $G$ either contains or complements $N$. This implies that all the
elements of 
$N\setminus\{1\}$ are $\equiv_\rem$ equivalent, and consequently 
$N$ is cyclic (of prime order). 

\noindent (3) Let $G$ be soluble and let $F=\fit(G).$ Since $\frat(G)=1,$ 
it follows from \cite[5.2.15]{Rob} that  $\fit(G)=\Soc(G)$, 
and hence $F=C_G(F)=\cap_{N\in \mathcal
  N}C_G(N),$
 where $\mathcal N$ is the set of the minimal normal subgroups of $G.$ 
But then $$\frac{G}{F}=\frac{G}{\bigcap_N C_G(N)}\leq \prod_N \Aut(N)$$ is abelian.
\end{proof}

	The conditions listed in the previous proposition are not
        sufficient to ensure that the relations $\equiv_\rem$ and
        $\equiv_\rc$ coincide on soluble groups $G$. 
In order to obtain a more precise
        result, let us fix some notation. Assume that
$G$ is soluble and satisfies the conclusions of Proposition~\ref{metab}. We set
	$F=\fit(G)$ and $Z=Z(G).$ Then
	$$F=V_1^{r_1}\times \cdots \times V_t^{r_t}\times Z,$$
where $V_1^{r_1},\dots,V_t^{r_t}$ are the non-central homogeneous components
of $F$ as a $G$-module. In particular, $V_i$ is cyclic of prime order
for every $i.$ 
Moreover $G=F :  H,$ where $H$ is a subdirect product of $\prod_i
H_i$, 
with $H_i \leq  \Aut(V_i).$ 
Finally, for $h=(h_1,\dots,h_t)\in H,$ define 
$\Omega(h)=\{i\in\{1,\dots,t\} \mid h_i=1\}.$ 

\begin{theorem}\label{thm:sol_c}
Let $G = F:H$ as above be a soluble group satisfying the conclusions of
Proposition~\ref{metab}. 
The relations $\equiv_\rem$ and $\equiv_\rc$ coincide on $G$ if and
only if the following property is satisfied, for all  $(z_1,h_1),
(z_2,h_2) \in Z\times H$ 
\begin{center} $(*)$ if
	$\langle (z_1,h_1) \rangle \frat H= \langle (z_2,h_2) \rangle \frat H$ and $\Omega(h_1)=\Omega(h_2),$ then
$\langle (z_1,h_1) \rangle=\langle (z_2,h_2) \rangle.$\end{center}
\end{theorem}

\begin{proof}
Let $x_1=(z_1,h_1)$, $x_2=(z_2,h_2) \in Z\times H,$ with
$h_1=(\alpha_1,\dots,\alpha_t)$ and $h_2=(\beta_1,\dots,\beta_t).$
Assume that  $\langle x_1 \rangle \frat H= \langle x_2 \rangle \frat
H$ and $\Omega(h_1)=\Omega(h_2)$. We claim that a maximal subgroup 
$M$ of $G$ contains $x_1$ if and only if it contains $x_2,$ and hence
that $x_1 \equiv_{\rem} x_2$.   

Let $W=V_1^{r_1}\times \cdots \times V_t^{r_t}$ and let $L=\frat(Z\times H)=\frat(H).$
If $W \leq M,$ then $W :  L \leq M,$ so 
$\langle x_i \rangle \subseteq M$ if and only if 
$\langle x_i\rangle L \subseteq M$. 
Since $\langle x_1\rangle L=
\langle x_2\rangle L$, we deduce that $x_1\in M$ if and only if
$x_2\in M.$
If $W \not\leq M$,  then there exists $i \in \{1,\dots,t\},$ 
a maximal $H$-invariant subgroup $U_i$ of $V_i^{r_i}$ and
$w_i \in V_i^{r_i}$ such that
	$$M=(V_1^{r_1}\times \cdots \times V_{i-1}^{r_{i-1}}\times U_i \times V_{i+1}^{r_{i+1}}\times \cdots \times V_{t}^{r_{t}}\times Z) :  H^{w_i}.$$
	Notice in particular that if $(\gamma_1,\dots,\gamma_r) \in H,$ then 
	$(\gamma_1,\dots,\gamma_r) \in M$ if and only if $\gamma_i \in
        U_iH_i^{w_i}$. In this case we can write $\gamma_i = u_i[w_i,
        h_i^{-1}]h_i = h_i$, so that
	$[w_i, \gamma_i^{-1}]\in
	U_i.$  
Since $V_i^{r_i}/U_i\cong_{H_i}V_i,$ we have that if $[w_i, \gamma_i^{-1}]\in U_i
	$
	then either $\gamma_i=1$ or $w_i\in U_i.$ If $w_i\in U_i$ then $x_1, x_2 \in M.$ So assume $w_i\not\in U_i.$ Since $\Omega(h_1)=\Omega(h_2),$ we have that
	$\alpha_i=1$ if and only only if $\beta_i=1,$ hence $x_1\in M$ if and only if $x_2\in M.$ We have proved that if $\equiv_\rem$ and $\equiv_\rc$ coincide, then $(*)$ holds.

For the converse, let
$x_1=w_1z_1h_1,$ $x_2=w_2z_2h_2$ be two elements of $G$ with
	$h_1,h_2 \in H, $ $z_1,z_2 \in Z$ and $w_1,w_2\in W.$ Assume that $x_1\equiv_\rem x_2.$ Since $w_1h_1$ and $h_1$ are conjugate in $G$, it is not
	restrictive to assume that $x_1=z_1h_1.$
 We claim that this implies that $w_2=1.$ Indeed, assume that
$w_2=(v_1,\dots,v_t)\neq 1.$ Then there exists an $i$ such that
$v_i\neq 1$, 
and consequently there exists
 a maximal $H$-invariant subgroup $U_i$ of $V_i^{r_i}$ with $v_i
 \notin U_i.$ 
This leads to a contradiction, since the maximal subgroup 	$$M=(V_1^{r_1}\times \cdots \times V_{i-1}^{r_{i-1}}\times U_i \times V_{i+1}^{r_{i+1}}\times \cdots \times V_{t}^{r_{t}}\times Z) :  H$$
	contains $x_1$ but not $x_2.$ 

Having $w_1=w_2=1,$ the argument used in the first part of this proof
shows that the condition $\Omega(h_1)=\Omega(h_2)$ is equivalent to 
saying that a maximal subgroup of $G$ not containing $W$ contains $x_1$
	if and only if it contains $x_2.$ On the other hand the maximal subgroups of $G$ containing $W$ are in bijective correspondence with those of $G/\frat H,$ hence 
	the condition $\langle x_1\rangle \frat H=\langle x_2\rangle
        \frat H$ is equivalent to saying
 that a maximal subgroup of $G$  containing $W$ contains $x_1$
	if and only if it contains $x_2.$ We have therefore
 proved that $x_1\equiv_\rem x_2$ implies
	that $\Omega(h_1)=\Omega(h_2)$ and $\langle x_1\rangle \frat H=\langle x_2\rangle \frat H,$ and therefore if $(*)$ holds, then $x_1\equiv_\rc x_2.$
\end{proof}

Here are two examples of groups which satisfy the conclusions of
Proposition~\ref{metab}, but do not satisfy condition $(*)$. Hence
$\equiv_\rc$-equivalence is finer than $\equiv_\rem$-equivalence.
\begin{enumerate}
\item Let $G$ be the sharply $2$-transitive group of degree~$17$, the
semidirect product of $C_{17}$ with a Singer cycle $C_{16}$. The maximal
subgroups are $C_{17}:C_8$ and the conjugates of $C_{16}$. In particular,
we see that elements of orders $2$, $4$ and $8$ in a fixed complement $C_{16}$
are all $\equiv_\rem$-equivalent. However, $\equiv_\rc$-equivalent elements
have the same order. 
\item A second example is $(\langle x \rangle  :  \langle y \rangle)\times \langle z \rangle$ with $|x|=19, |y|=9, |z|=3$ (indeed $(y^3,z)\equiv_\rem (y^6,z)$).
\end{enumerate}

\begin{prop}\label{prop:min_norm}
Assume that a finite group $G$ contains a minimal normal subgroup $N=S_1\times \cdots \times S_t$, with $S_i\cong S$ a finite
nonabelian simple group. If either $t\geq 3$, or
$t=2$ and $S$ is not isomorphic to $\mathrm{P\Omega}^+_8(q)$ with
$q=2$ or $3$,
then the relations $\equiv_\rem$ and $\equiv_\rc$ do not coincide on $G.$
\end{prop}

\begin{proof}
It is standard (see, for example, \cite[Remark 1.1.040]{BBE}) that 
if a maximal subgroup $M$ of $G$ does not contain
$N$,
 then one  of the following occurs:
\begin{enumerate}
\item[(1)] $M\cap N=1;$
\item[(2)] $M$ is of \emph{product type}: in this case there exist  
$\alpha_2,\dots,\alpha_t \in \Aut(S),$ independent of 
the choice of $M$, $s_2,\dots,s_t \in S$ and a proper subgroup $K$ of $S$ such that
$M\cap N \leq K\times K^{s_2\alpha_2}\times\cdots\times K^{s_t\alpha_t};$
\item[(3)] $M$ is of \emph{diagonal type}: in this case there exists a partition $\Phi:=\{B_1,\dots,B_u\}$ of $\{1,\dots,t\}$ into blocks of the same size such that $M\cap N \leq \prod_{B\in \Phi}D_B$ where $D_B$ is a full diagonal subgroup of $\prod_{j\in B}S_j.$
	\end{enumerate}
By \cite[Theorem 5.1]{ig} or \cite[Theorem 7.1]{gm}, there 
exist $a, b \in S$  with the property that $\langle
a^\gamma,b^\delta\rangle=S$ 
for each choice of $\gamma, \delta \in  S.$ Moreover if 
$S\neq \mathrm{P\Omega}^+_8(q),$ $q=2$ or $3,$ then
$a$ and $b$ are not conjugate in $\Aut(S).$

Let $x,y \in S$ and consider 
$$g_{x,y}=\begin{array}{ll}(a^{x},b^{y\alpha_2},a,\dots,a,1)&  {\mbox{if } }t>2\\
(a^{x},b^{y\alpha_2})&  {\mbox{otherwise.}}\end{array}$$

There is no maximal subgroup of product type containing 
$g_{x,y}.$  Otherwise we would have $a^{x}\in K,$  $b^{y\alpha_2}\in
K^{s_2\alpha_2},$ hence $S=\langle a^{x}, b^{y s_2^{-1}}\rangle\leq
K,$ contradicting the fact that $K$ is a proper subgroup of $S$. 
Moreover, since either $t\geq 3$ or $a$ and $b$ are
not conjugate
 in $\Aut(S),$ 
 no maximal subgroup of diagonal type contains  $g_{x,y}.$ Therefore 
 $g_{x,y}\in M$ if and only if $N\leq M$, for all maximal subgroups $M$. 
Hence, all the elements of the subset $\{ g_{x,y} \, \mid \,
x,y \in S\}$ are $\equiv_\rem$ equivalent, and therefore the relations
$\equiv_\rem$ and $\equiv_\rc$ do not coincide on $G.$
\end{proof}

\begin{cor}
Let $G$ be a finite group. 
If the relations $\equiv_\rem$ and $\equiv_\rc$  coincide on $G,$ then $G/\Soc(G)$ is soluble.
\end{cor}
\begin{proof}
Since the relations $\equiv_\rem$ and $\equiv_\rc$  coincide,
$\frat(G)=1$ by Proposition~\ref{metab}(1),
 and consequently $\Soc(G)=F^*(G),$ where $F^*(G)$ is
the generalized Fitting subgroup of $G$.

Let $F^*(G)=Z(G)\times N_1\times \dots \times N_t,$ where
$N_1,\dots,N_t$ are non-central minimal normal subgroups. Since
$Z(G)=C_G(F^*)=\bigcap_i C_G(N_i),$ we have $G/Z(G)\leq \prod_i
G/C_G(N_i).$ To conclude, notice that if $N_i$ is abelian, then $N_i$
is cyclic and $G/C_G(N_i)$ is abelian, while if $N_i$ is nonabelian,
then by Proposition~\ref{prop:min_norm} the group $N_i\cong S_i^{t_i}$ with $t_i\leq 2$ and $G/(N_iC_G(N_i))\leq\out S \wr \perm(t_i),$ which is soluble.
\end{proof}

\begin{prob}
Find an equivalence relation that is easier to calculate than
$\equiv_\rem$, but coarser than $\equiv_\rc$. Determine for which
insoluble groups $G$ the relations $\equiv_{\rem}$ and $\equiv_{\rc}$
coincide.
\end{prob}

\subsection{Asymptotics and enumeration}

We now briefly suggest some directions for further study of the
asymptotics of our new relations.

\begin{prop}
Let $G$ be $\sym_n$ or $\alt_n$. Then for almost all elements 
$x,y\in G$ (all but a
proportion tending to $0$ as $n\to\infty$), the following are equivalent:
\begin{enumerate}\itemsep0pt
\item[(1)] $x\equiv_\rem y$;
\item[(2)] $x\equiv_\rem^{(2)} y$;
\item[(3)] 
the cycles of $x$ and $y$ induce the same partition of $\{1,\ldots,n\}$.
\end{enumerate}
\end{prop}

\begin{proof} This depends on a  theorem of {\L}uczak and Pyber~\cite{lp},
which states that for almost all $x\in \sym_n$, the only transitive subgroups
of $\sym_n$ containing $x$ are $\sym_n$ and (possibly) 
$\alt_n$. We restrict our
attention to these elements $x$.

Consider first the case where $G=\sym_n$. Then, apart from $\alt_n$, 
the maximal
subgroups containing $x$ are of the form $\sym_k\times \sym_{n-k}$, where the two
orbits are unions of cycles of $x$. Moreover, the cycle lengths determine
whether or not $x\in \alt_n$. So (1) and (3) are equivalent.

In addition, for all $z \in G$, we see that
 $\langle x,z\rangle = G$ whenever $\langle x, z \rangle$ is
transitive, and $z\notin \alt_n$ if it happens that
$x\in \alt_n$. Membership of this set is also determined by the cycles
of $x$: the 
transitivity condition requires that the hypergraph whose edges are the cycles
of $x$ and $z$ is connected. So (2) is also equivalent to (3).

If $G=\alt_n$, then only simple modifications are required; the argument is
simpler because no parity conditions are necessary.
\end{proof}

Shalev in \cite{shalev} proved a similar result  for $\gl_n(q)$ to 
{\L}uczak and Pyber's result for $\sym_n$: a
  random element of $\gl_n(q)$ lies in no proper irreducible subgroup
  not containing $\slin_n(q)$. This could be used to prove a similar
  statement for groups lying between $\psl_n(q)$ and $\pgl_n(q)$. 

\begin{qn} 
Are there only
 finitely many finite almost simple 
groups on which the relations $\equiv_{\rem}$ and $\equiv_{\rc}$
coincide? 
\end{qn}

Another very natural question is: how many $\equiv_\rc$- and
$\equiv_\rem$-classes 
are there in the
symmetric group $\sym_n$? The numbers of $\equiv_\rc$-classes
in the symmetric groups $\sym_n$ form sequence A051625 in the On-line
Encyclopedia of Integer Sequences~\cite{oeis}.
The sequence of numbers of $\equiv_\rem$ classes, which begins
\[1,2,5,15,67,362,1479,12210,\ldots\]
has recently been added to
the OEIS, where it
appears as Sequence A270534.

If we cannot find a formula for these sequences, can we say anything about their
asymptotics? We saw above that, for almost all elements of $\sym_n$, the
$\equiv_\rem$-equivalence class is determined by the cycle partition, 
which might suggest
that the sequence grows like the Bell numbers (sequence A000110 in the OEIS).
However, the elements not covered by this theorem can destroy this estimate.

For example, let $p$ be a prime such that the only insoluble
transitive groups of degree $p$ are the symmetric and alternating groups.
Then the above analysis applies to all elements whose cycle type is not
a single $p$-cycle or a fixed point and $l$ $k$-cycles (where $1+kl=p$). It
is easy to show that two elements $x$ and $y$ 
with one of these excluded cycle types satisfy $x \equiv_{\rem}y$
if and only if they satisfy $x \equiv_{\rc} y$. So there
are $(p-2)!$ equivalence classes of $p$-cycles, for example; this number is
much greater than the $p$th Bell number. (In this special case, we can write
down a formula for the number of $\equiv_\rem$-equivalence classes.)

\section{The generating graph of a group}\label{sec:gen}

In the remainder of the paper, we use the relations that we have
defined to study an object of general interest, the generating graph
of a finite group.
\begin{defn}
The \emph{generating graph} of a finite group $G$ is the graph with vertex
set $G$, in which two vertices $x$ and $y$ are joined if and only if
$\langle x,y\rangle=G$. 
\end{defn}

Of course this graph is null unless $G$ is
$2$-generated. We adopt the convention that, if the group is cyclic, then
any generator of the group carries a loop in the generating graph.

A useful concept when studying the generating graph is the spread of a
group.

\begin{defn}\label{def:spread}
A group $G$ has \emph{spread $k$} if $k$ is the largest number such
that for any set $S$ of
$k$ nonidentity elements, there exists $x$ such that $\langle x,s\rangle=G$
for all $s\in S$. 
\end{defn}

Thus the spread is nonzero if and only if no vertex
of the generating graph except the identity is isolated; and spread at least
$2$ implies diameter at most~$2$. 

Among the graph-theoretic invariants which have been studied for this graph
are the following.
\begin{enumerate}\itemsep0pt
\item[(1)] The spread.
\item[(2)] The \emph{clique number}: the largest size of a set of group
elements, any two of which generate the group. 
\item[(3)] The \emph{chromatic number}:  the smallest number of parts in a
partition of the group into subsets containing no $2$-element generating set.
\item[(4)] The \emph{total domination number}: the smallest size of a set $S$ with
the property that, for any element $x$, there exists $s\in S$ such that $x$
and $s$ generate the group.
\item[(5)] The isomorphism type: if $\Gamma(G) \cong \Gamma(H)$ for two
  groups $G$ and $H$, then when is $G \cong H$?
\end{enumerate}

\begin{defn}
In any graph $X$, we can define an equivalence
relation $\equiv_\rg$ by the rule $x\equiv_\rg y$ if $x$ and $y$
have the same set of neighbours in the graph. (Think of $\rg$ as meaning
``graph'', or ``generating'' if we are thinking of the generating graph.)
Then we define a \emph{reduced graph} $\overline{X}$ whose vertices are
the $\equiv_\rg$-classes in $X$, two classes joined in $\overline{X}$ if their
vertices are joined in $X$. 
\end{defn}

Alternatively, we can take the vertex set to be
any set of equivalence class representatives, and the graph to be the induced
subgraph on this set. (The term ``reduced graph'' was used by
Hall~\cite{copolar} in his work on copolar spaces, and consequentially we
term the process of producing it ``reduction''; but we warn readers that
the term ``graph reduction'' has a very different meaning in computer science.)

The reduction process preserves the graph parameters noted above:

\begin{prop}
The clique number, chromatic number, total domination number, and spread of
the generating $\Gamma(G)$ are equal to the corresponding parameters
of the reduced generating graph $\overline{\Gamma}(G)$. Furthermore,
if $\Gamma(G) \cong \Gamma(H)$ then $\overline{\Gamma}(G) \cong
\overline{\Gamma}(H)$. 
\label{p:g_params}
\end{prop}

\begin{proof}
Clear.
\end{proof}

The following is immediate from
the definition of $\equiv_\rem^{(r)}$.

\begin{prop}
Let $G$ be a finite group. Then 
the relations $\equiv_\rg$ on $\Gamma(G)$ and $\equiv_\rem^{(2)}$ on
$G$ coincide; hence
$\equiv_\rem$ is a refinement of $\equiv_\rg$, and is equal to
$\equiv_\rg$ if and only if $\psi(G) \leq 2$.  
\end{prop}
Hence, in what follows, we shall write $\equiv_{\rg}$ to denote
$\equiv_{\rem}^{(2)}$. 

Recall Definition~\ref{def:efficient} of efficient generation.

\begin{theorem}\label{thm:spread_phi}
Let $G$ be a finite group with $d(G) = 2$. 
\begin{enumerate}
\item[(1)] $G$ has nonzero spread
if and only if $G$ is efficiently generated and has trivial
Frattini subgroup. 
\item[(2)] If $G$ is soluble and has nonzero spread, then $\psi(G) =
  2$.
\end{enumerate}
\end{theorem}

\begin{proof}
(1) Since the spread of $G$ is nonzero, every nonidentity element of $G$ lies in a
$2$-element generating set of $G$, so $d_{x}(G) = 1$ unless $x =
1$. Hence $G$ is efficiently generated and $\frat(G) = 1$. The
converse is clear.

(2) By Part (1), the assumption that $G$ has nonzero spread implies
that $G$ is efficiently generated. Hence from
Corollary~\ref{psisol}, we see that $\psi(G) = d(G) = 2$. 
\end{proof}

Notice that it is immediate from Theorem~\ref{thm:spread_phi} that if
$G$ is a $2$-generator group of spread $0$ and trivial Frattini subgroup, then
$\psi(G) \geq 3$. 
For example, 
double transpositions are isolated vertices in
$\Gamma(\sym_4)$, and so are equivalent to the identity under $\equiv_\rg$,
though clearly not under $\equiv_\rem$. In fact this group has fourteen
$\equiv_\rg$-classes but fifteen
$\equiv_\rem$-classes, and as previously noted $\psi(\sym_4)=3$.

We shall therefore proceed for much of the following section
by restricting to groups with nonzero
spread, despite that fact
that we 
don't know whether
Theorem~\ref{thm:spread_phi}(2)
is also true 
without the solubility assumption.

\begin{conj}
Let $G$ be a finite group of nonzero spread. Then $\psi(G) \leq 
2$. 
\end{conj}

By Lemma~\ref{lemma2}, if $G$ is a group with nonzero spread, then
$\psi(G) = 2$ whenever for all maximal subgroups $M$, 
and for all  $x\notin M$, there exists $z\in M$ such that $\langle x,z\rangle=G$.
This approach can be applied to $\sym_5$, $\psl_2(7)$, and
$\psl_2(11)$. However, it fails in the case of  $\alt_5$ 
with respect to the smallest
maximal subgroups (isomorphic to $\sym_3$). It also fails for $\psl_2(q)$
for $q=8,9,13$, even though $\psi(G) = 2$ for all of these groups. 



\section{Automorphism groups}\label{sec:aut}

A striking thing about generating graphs 
is that they have huge automorphism groups, and these groups are
poorly understood. For example, 
the automorphism group of the generating graph of the alternating
group $\alt_5$ has order
$2^{31}3^75$.

The reason is simple. Any nontrivial 
element of $\alt_5$ has order $2$, $3$ or $5$. An
element of order $3$ or $5$ can be replaced by a nonidentity power of itself
in any generating set. Thus the sets of nonidentity powers can be permuted
arbitrarily, and we find a group of order $2^{10}(4!)^6=2^{28}3^6$ of
automorphisms fixing these sets. The quotient has order $120$ and is isomorphic
to $\Aut(\alt_5)=\sym_5$.

Hence, for $G=\alt_5$, the automorphism group of the generating graph
$\Gamma(G)$ has a normal subgroup which is the direct
product of symmetric groups on the $\equiv_\rg$-classes, and the quotient is
the automorphism group of the reduced graph $\overline{\Gamma}(G)$.
In general, a similar statement holds, but to state it we require one
further definition. 

\begin{defn}
We define a \emph{weighting} of the reduced generating graph, by assigning
to each vertex a weight which is the cardinality of the corresponding
$\equiv_\rg$-class. Now let $\overline{\Gamma}_\rw(G)$ denote the weighted
graph, and let $\Aut(\overline{\Gamma}_\rw(G))$ be the group of
\emph{weight-preserving automorphisms} of $\overline{\Gamma}_\rw(G)$. 
\end{defn}

Note that the restriction to $\Aut(\overline{\Gamma}_\rw(G))$ is
necessary, as in general
 an automorphism of $\overline{\Gamma}(G)$ can fail to lift to an
automorphism of $\Gamma(G)$. 
For an example of this, take $G =
\psl_2(16)$. Then $\Aut(\overline{\Gamma}(G)) \cong 2 \times
\Aut(\psl_2(16))$. However, the central involution interchanges elements of
order $3$ with elements of order $5$. The $\equiv_{\rem}$-class of the elements of
order $3$ has size $2$, and contains only the elements and their
inverses. However, the $\equiv_{\rem}$-class of elements of order $5$ has size $4$
(it clearly contains all nontrivial elements of the 
cyclic subgroup, but in fact contains
no more than this).

The following theorem shows that to describe the automorphism
group of $\Gamma(G)$, it suffices to know the multiset of sizes of the
$\equiv_{\rg}$-classes of $G$, and the automorphism group of
$\overline{\Gamma}_{\rw}(G)$. 

\begin{theorem}\label{thm:aut_g}
Let the $\equiv_\rg$-classes of a finite group
$G$ be of sizes $k_1, \ldots, k_n$. Then 
$$A:= \Aut(\Gamma(G)) = \left(\sym_{k_1} \times \cdots \times \sym_{k_n} \right) \, : \,
\Aut(\overline{\Gamma}_\rw(G)).$$
\end{theorem}

\begin{proof}
Let $N:= \prod_{i = i}^n \sym_{k_i}$. First we show that $N \leq A$,
then that $A$ is an extension of $N$ by a subgroup of 
$\Aut(\overline{\Gamma}_\rw(G))$, and finally that the whole of
$\Aut(\overline{\Gamma}_\rw(G))$ is induced by $A$,  and the extension
splits.

For the first claim, let $x, y \in G$ such that $x \equiv_{\rg}
y$. Then for all $z \in G$, there is an edge from $x$ to $z$ if and
only if there is an edge from $y$ to $z$. Hence the map interchanging
$x$ and $y$ and fixing all other vertices in $\Gamma(G)$ is an
automorphism of $\Gamma(G)$, so $N \leq A$.

For the second, we show that $A$ acts on the $\equiv_{\rg}$-classes of
$\Gamma(G)$. For $z \in G$, write $N(z)$ for the set of neighbours of $z$
in $\Gamma(G)$. Suppose that $x \equiv_{\rg} y$, as before. Then for all
 $a \in A$  we see that $$N(x^a) = N(x)^a = N(y)^a = N(y^a),$$ and so
 $x^a \equiv_{\rg} y^a$, as required. Hence $A$ is an extension of $N$
 by a subgroup of $\Aut(\overline{\Gamma}_\rw(G))$.

For the final claim, fix an ordering of the elements in each
$\equiv_{\rg}$ class of $G$, and identify the vertices of $\Gamma(G)$ with the
ordered pairs $\{(i, j) \ : \ 1 \leq j \leq n,  \ 1 \leq i \leq
k_j\}$. Let $\sigma \in \Aut(\overline{\Gamma}_{\rw}(G))$, and let
$j_1, j_2$ be adjacent vertices in $\overline{\Gamma}_{\rw}(G)$, so
that $j_1^\sigma$ and $j_2^\sigma$ are also adjacent. Then $k_{j_1} =
k_{j_1^\sigma}$, and for $1 \leq i \leq k_{j_1}$ vertex $(i, j_1)$ is
adjacent to vertex $(i, j_2)$. Hence we can define $\tau$ to be the
map sending $(i, j)$ to $(i, j^\sigma)$, and then $\tau \in
\Aut(\Gamma(G))$ induces $\sigma$. The result follows.
\end{proof}


Note that $\Aut(G)$ preserves the generating graph $\Gamma(G)$, and
hence
automorphisms of $G$ permute the $\equiv_\rg$-classes.
We define $\Aut^*(G)$ be the group induced by $\Aut(G)$ on
$\overline{\Gamma}(G)$. The following is clear. 

\begin{prop}
Let $G$ be a group with $d(G) \leq 2$. Then
\[\Aut^*(G)\le \Aut(\overline{\Gamma}_\rw(G))\le\Aut(\overline{\Gamma}(G)).\]
\end{prop}


In the remainder of the paper we shall analyse these three
automorphism groups,
concentrating on the groups $G$ with nonzero spread.
Such a group $G$ has no non-cyclic proper quotients. Moreover (see
for example \cite{LMRD}), it satisfies
one of the following:
\begin{enumerate}\itemsep0pt
\item[(1)] $G$ is cyclic;
\item[(2)] $G\cong C_p\times C_p$ for some prime $p$;
\item[(3)] $G$ is the semi-direct product of its unique minimal normal subgroup
$N$ (which is elementary abelian) by an irreducible subgroup $C$ of a Singer
cycle acting on $N$;
\item[(4)] $G$ has a normal subgroup $N\cong T_1\times\cdots\times T_r$,
where $T_1,\ldots,T_r$ are isomorphic nonabelian simple groups; $G/N$ has
order $rm$ for some $m$ dividing $|\mathrm{Out}(T_1)|$, and induces a cyclic
permutation of the factors.
\end{enumerate}

We shall show that $\Aut^*(G)$  is trivial for groups of
type (1), and is equal to $\Aut(G)$ for groups of type
(3) and (4). Furthermore, we shall show that in type (1) there is 
a spectacularly large gap between  $\Aut(\overline{\Gamma}(G))$ 
and $\Aut(\overline{\Gamma}_{\rw}(G))$, whilst in type (2) and
(3) we find that $\Aut^*(G) \neq \Aut(\overline{\Gamma}_\rw(G))$. 

First we consider the groups of type (1).

\begin{prop}\label{p:cyclic}
Let $G$ be the cyclic group of order $n=p_1^{a_1}p_2^{a_2}\cdots
p_r^{a_r}$. Then $\overline{\Gamma}(G)$ has $2^r$ vertices. The group
$\Aut^*(G)=\Aut(\overline{\Gamma}_\rw(G))$ is trivial, while
$\Aut(\overline{\Gamma}(G)) \cong \sym_r$.
Hence $\Aut(\Gamma(G))=\prod_{I\subseteq\{1,\ldots,r\}}\sym_{n_I}$,
where $$n_I= \frac{n}{p_1p_2 \cdots p_r} \prod_{i \in I} (p_i - 1).$$ 
\end{prop}
\begin{proof}
First, vertices in the same coset of the Frattini subgroup $\Phi(G)$ get
identified when we reduce the generating graph, and the weights are
multiplied
by $|\Phi(G)| = \frac{n}{p_1 \cdots p_r}$.
 So we can assume that the Frattini
subgroup is trivial, that is, $n= p_1p_2\cdots p_r$.

We know that in this case the $\equiv_{\rg}$- and 
$\equiv_{\rem}$-relations coincide,
and it is more convenient to use the latter. The group has $r$ maximal
subgroups (one of index $p_i$ for each $i$) and the lattice of their
intersections is the lattice of subsets of $\{1,\ldots,r\}$. So, for any subset
$I$ of $\{1,\ldots,r\}$, there is a unique vertex $v_I$ of the reduced graph
corresponding to the intersection of the subgroups of index $p_i$ for $i\in I$;
and $v_I$ is joined to $v_J$ if and only if $I\cap J=\emptyset$.

We claim that the automorphism group of $\overline{\Gamma}(G)$
 is the symmetric group
$\sym_r$. It is clear that $\sym_r$ acts as automorphisms of the 
graph; it suffices
to prove that there are no more.

There is a unique vertex $v_\emptyset$ joined to all others. Apart from this
vertex, there are $r$ vertices whose neighbour sets are maximal with respect
to inclusion, namely $v_{\{i\}}$ for $i=1,\ldots,r$, which must be permuted
by the automorphism group. It suffices to show that only the identity fixes
all these vertices. But any further vertex is uniquely specified by its
neighbours within this set: $v_I$ is joined precisely to $v_{\{j\}}$ for
$j\notin I$.

What is the subgroup of $\sym_r$ fixing the weights? Recall that the weight of
a vertex $v_I$ is the number of elements of $G$ which are equivalent to this
vertex of the reduced graph, that is, which lie in the maximal subgroups
of index $p_i$ for $i\in I$ and no others. This is the number of generators
of the intersection of these maximal subgroups, which is
\[\prod_{j\notin I}(p_j-1).\]
Now it can happen that two of these weights are equal, even for elements in
the same $\sym_r$-orbit. (For example, let $n=2.3.7.13=546$. The subgroups of
orders $2.13$ and $3.7$ each have $12$ generators.)

However, only the identity element of $\sym_r$ preserves all the weights. For
the minimal nonidentity elements $C_{p_i}$ have distinct weights $p_i-1$,
and so all are fixed by the weight-preserving subgroup.
\end{proof}

\begin{prop}\label{prop:square}
Let $G \cong C_p^2$. Then $\overline{\Gamma}(G)$ has $p+2$ vertices,
with $\Aut(G) \cong\gl_2(p)$ and $\Aut^\ast(G)
\cong\pgl_2(p)$. On the other hand, $\Aut(\overline{\Gamma}(G))$
and $\Aut(\overline{\Gamma}_\rw(G))$ are both
 isomorphic to $\sym_{p+1}$, fixing the isolated vertex corresponding
 to the identity. 
Furthermore, 
the group $\Aut(\Gamma(G))=\sym_{p-1}\wr \sym_{p+1}$.
\end{prop}

\begin{proof}
Thinking of $G$ as a vector space, two nonidentity elements $x, y \in
G$ fail to
generate $G$ if and only if
they lie in the same $1$-dimensional subspace. Furthermore, they lie
in the same $1$-dimensional subspace if and only if $x \equiv_{\rg}
y$.  Thus $\overline{\Gamma}(G)$ is the disjoint union of the complete
graph $K_{p+1}$ and a vertex representing the identity, and all 
weights in $K_{p+1}$ are equal to $p-1$.
\end{proof}

Before considering the groups of type (3), we require a standard
graph-theoretic definition.

\begin{defn}
The \emph{categorical product} $X\times Y$ of two
graphs $X$ and $Y$ is the graph whose vertex set is the cartesian product
of the vertex sets, with $(x_1,y_1)$ joined to $(x_2,y_2)$ if and only if
$x_1$ is joined to $x_2$ in $X$ and $y_1$ is joined to $y_2$ in $Y$.
\end{defn}

\begin{prop}\label{prop:case_c}
Let $G \cong C_p^k : C_n$ be nonabelian with all proper quotients
cyclic, and let $n = p_1^{a_1}p_2^{a_2}\cdots
p_r^{a_r}$. 
The graph $\overline{\Gamma}(G)$ has $(2^r-1)p^k+2$ vertices 
if $n$ is squarefree, and $2^rp^k+2$ otherwise. The groups
$\Aut(G)$ and $\Aut^*(G)$ are both isomorphic to $C_p^k : \gaml_1(p^k)$.
Furthermore, 
$\Aut(\overline{\Gamma}_\rw(G))
\cong \sym_{p^k}$, whilst $\Aut(\overline{\Gamma}(G)) \cong \sym_{p^k}
\times \sym_r$. 
\end{prop}

\begin{proof}
The elementary abelian subgroup $C_p^k$ is characteristic in $G$, so 
$\Aut(G) \leq\agl_k(p)$.  The cyclic subgroup must embed as an irreducible
subgroup of a Singer cycle, and so its centraliser in $\gl_k(p)$ is the full
Singer cycle $C_{p^k-1}$, and its normaliser is the normaliser of the Singer
cycle, which is $\gaml_1(p^k)$. 

We claim that $\overline{\Gamma}(G)$ is obtained from the categorical product
of $\overline{\Gamma}(C_n)$ and the complete graph $K_{p^k}$ by
the following procedure: 
\begin{enumerate}\itemsep0pt
\item[(1)]
\begin{enumerate}
\item[(a)] If $n$ is squarefree, identify all the vertices whose first component
corresponds to the identity in $C_n$.
\item[(b)] Otherwise, add a vertex adjacent to all vertices whose first component
corresponds to a generator in $C_n$. 
\end{enumerate}
The vertex in either case corresponds
to the nonidentity elements of the minimal normal subgroup of $G$.
\item[(2)] Then add an isolated vertex corresponding to the identity.
\end{enumerate}
Note that generators of $C_n$ carry loops in $\Gamma(C_n)$; these give
rise to edges in the categorical product between any two elements whose first
components are equal and correspond to generators of $C_n$.

The weights of the vertices are the weights of their first components in
$\overline{\Gamma}(C_n)$, except for the identified or added vertex in 
Step (1), whose weight is $p^k$ in case (1)(a) and
$p^k(|\Phi(C_n)|-1)$ in case (1)(b), and the identity which has weight
$1$. 

Now we demonstrate that this structure is correct.

First note that in $\Gamma(G)$ all the nonidentity 
elements of the normal subgroup $C_p^k$ are adjacent
to all (and only) the generators of the complements
$C_n$; so they all have the same neighbour sets and are 
$\equiv_\rg$-equivalent.
Elements outside the normal subgroup are joined if and only if
they lie in a different complements and their images in the $C_n$ quotient
generate $C_n$. 
So two such elements are $\equiv_\rg$-equivalent if they lie in the
same complement and are $\rg$-equivalent in $C_n$. Thus the graph
has the structure claimed.

We now use the results of Proposition~\ref{p:cyclic}, from which
the number of vertices of $\overline{\Gamma}(G)$
follows immediately. 
The automorphism
group of $\overline{\Gamma}(C_n)$ is $\sym_r$, 
so $\Aut(\overline{\Gamma}(G))$ 
is $\sym_{p^k}\times \sym_r$.

Conversely,   the group $\Aut(\overline{\Gamma}_{\rw}(C_n))$ is
trivial, so 
the weight-preserving automorphisms of $\overline{\Gamma}(G)$ are just
the permutations of the $p^k$ vertices of the complete graph.

Finally, we prove the claims about $\Aut^*(G)$. If $\Aut^*(G) \neq
\Aut(G)$, then the unique minimal normal subgroup $C_p^k$ of $\Aut(G)$ 
must act trivially on $\overline{\Gamma}_{\rw}(G)$. However, this is
not possible, for the following reason: let $g$ be any element of $G$
that generates a complement to $C_p^k$ in $g$, and let $x$ be any
nontrivial element of $C_p^k$. Then $\langle g \rangle$ is a maximal
subgroup of $G$, so $g^x \not\in \langle g \rangle$ and $\langle g,
g^x \rangle = G$. Hence $g$ and $g^x$ are incident in $\Gamma(G)$, and
so $g \not\equiv_{\Gamma} g^x$. Hence $x$ acts nontrivially on
$\Gamma(G)$. 
\end{proof}

For groups $G$ as in the previous result, the kernel of the
homomorphism from $\Aut(\Gamma(G))$ to 
$\Aut(\overline{\Gamma}_w(G))$ is the direct product of symmetric
groups whose degrees are implicit in the proof: $p^k-1$ once, and
the sizes of the nontrivial $\equiv_\Gamma$-classes in $C_n$ (which can be read off from
Proposition 4.7) each $p^k$ times. 
The action of $\sym_{p^k}$ is to permute the factors apart from the
$\sym_{p^k-1}$.

\begin{ex}
Consider the case $G=C_5:C_4$. The generating graph
for $C_4=\langle x\rangle$ is the complete graph $K_4$ with the edge
$\{1,x^2\}$ deleted and loops at $x$ and $x^3$. So the reduced graph 
identifies $1$ and $x^2$, and also $x$ and $x^3$, and is an edge with a loop
at one end. Thus, the reduced generating graph for $C_5:C_4$ has $12$ vertices,
say $a_1,\ldots,a_5,b_1,\ldots,b_5,c,d$, with all edges $\{a_i,a_j\}$,
all edges $\{a_i,b_j\}$, and no edges $\{b_i,b_j\}$ for $i\ne  j$, all
edges $\{a_i,c\}$, and $d$ isolated. (Here $a_i$ corresponds to an inverse
pair of elements of order~$4$, $b_i$ to an element of order~$2$, $c$ to
the four elements of order~$5$, and $d$ to the identity.) Here the
kernel of the homomorphism from $\Aut(\Gamma(G))$ to
$\Aut(\overline{\Gamma}_\rw(G))$ is 
$\sym_4\times (\sym_2)^5$. 
\end{ex}

It remains to perform the
analysis for the groups of type (4).

\begin{theorem}\label{thm:wreath}
Let $T$ be a finite simple group and let 
$N=T^r\leq G \leq \Aut(T) \wr \langle \sigma\rangle$, where $\sigma$
acts as an $r$-cycle. Assume that 
there exists $g=(y_1,\dots,y_r)\sigma$, with 
$y_1,\dots,y_r\in \Aut(T)$, such that
$G=N\langle g \rangle.$   By substituting $g$ by a conjugate in
$\Aut(T) \wr \langle \sigma\rangle$, if necessary,
 we may assume that $g=(y,1,\dots,1)\sigma.$
If there exist $s, t \in T$ such that
 $T\leq \langle ys, (ys)^t\rangle,$
then $\Aut(G) = \Aut^\ast(G)$.
\end{theorem}

\begin{proof}
   Since $N$ is the unique minimal normal subgroup of $\Aut(G)$, if the
   conclusion is false, then $N$ must act trivially on $\overline{\Gamma}(G)$.
   But this is impossible, for the following reason.

 Let $\bar y=ys$ and  $\bar g=(\bar y,1,\dots,1)\sigma \in G$. 
Notice that $G$ contains $\bar g^r=(\bar y,\dots,\bar y)$,
$z=(t,1,\dots,1)$ and  $(\bar g^r)^z=(\bar y^t,\bar y,\dots,\bar y).$ Consider the subgroup $X$ of $G$ 
 generated by  $\bar g$ and  $(\overline g^r)^z.$
Since $X$ contains $(\bar y,\dots,\bar y)$ and $(\bar y^t,\bar y,\dots,\bar y)$, 
we easily conclude that 
   $X = G=\langle \bar g, (\overline g^r)^z\rangle.$
   Now if $N$ acts trivially, then conjugacy classes under $N$ are
   contained in $\equiv_\Gamma$-equivalence classes. 
Hence, in particular,  $\bar g^r \equiv_{\rg} (\bar g^r)^z$, 
   so $G=\langle \bar g, (\bar
   g^r)^z\rangle=\langle \bar g, \bar g^r\rangle=\langle \bar g
   \rangle$, a contradiction.
\end{proof}



\begin{theorem}
Let $G$ be a group of nonzero spread. Then $\Aut^*(G) =
\Aut(G)$ if and only if $G$ is nonabelian. 
\end{theorem}

\begin{proof}
The abelian groups of nonzero spread were considered in
Propositions~\ref{p:cyclic} and \ref{prop:square}, where we showed
that
$\Aut^*(G) \neq \Aut(G)$.

The soluble nonabelian groups of nonzero spread were considered 
in Proposition~\ref{prop:case_c}, where we showed that $\Aut^*(G) =
\Aut(G)$. 

The only remaining case is the insoluble groups of nonzero spread
(that is type (4)), so
let $G$ be such a group, 
and let $N\cong T^r=\Soc(G)$.  We can identify $G$ with a 
subgroup of $\Aut(T) \wr \langle \sigma \rangle,$ where
$\sigma$ is the $r$-cycle $(1,2,\dots,r).$ Let $t$ be an involution in 
$T$ and let $n=(t,1,\dots,1).$ Since $G$ is of nonzero spread, 
there exists $g\in G$
with $G=\langle n, g\rangle.$ Up to conjugation by an element of 
$(\Aut T)^r,$
we may assume $g=(y,1,\dots,1)\sigma$ for some $y 
\in \Aut(T)$. 
But now $G=\langle n,g\rangle$ implies that 
$H=\langle y,t\rangle$ is almost simple with socle
$T$. Since $|t|=2,$ the subgroup $\langle y, y^t\rangle$ is normal in
$H$.
From this we see that
 $T\leq \langle y, y^t\rangle$, and so by Theorem~\ref{thm:wreath}, we
 conclude that
$\Aut(G)=\Aut^*(G).$
\end{proof}

We finish this discussion with an open problem: 
\begin{qn}
Let $G$ be an insoluble group of nonzero spread. 
Is $\Aut(G) = \Aut(\overline{\Gamma}_{\rw}(G))$?
\end{qn}

We know of no examples where this is not the case.

\subsection{Calculations with $\overline{\Gamma}_{\rw}(G)$}

In this subsection we describe some experiments that we have carried
out on insoluble groups with nonzero spread.

Recall the definition of the m-universal action from
Subsection~\ref{subsec:calculate}, and that we showed in
Theorem~\ref{thm:small_as} that if $G$ is almost simple, with socle of
order less than $10000$ and all proper quotients cyclic then $\psi(G)
= 2$. It is immediate from Lemma~\ref{lem:m-universal}(2) 
that two group elements $x,y$ are incident in $\Gamma(G)$
if
and only if the fixed-point sets of $x$ and $y$ in the m-universal action
are disjoint.

For each such almost simple group $G$, we constructed
$\overline{\Gamma}(G)$ and hence $\Aut(\overline{\Gamma}(G))$. For all
such groups except for $\psl_{2}(16)$ and $\psl_{2}(25)$ we found that
$\Aut(\overline{\Gamma}(G)) \cong \Aut(G)$. In these remaining two
cases, $\Aut(\overline{\Gamma}(G)) \cong C_2 \times \Aut(G)$, but the
elements in the centre of $\Aut(\overline{\Gamma}(G))$ do not preserve
the graph weightings. From this we can conclude:

\begin{theorem}
Let $G$ be an almost simple group with  socle of order less than 10000
such that all proper quotients of $G$ are cyclic. Then 
$\Aut(\overline{\Gamma}_{\rw}(G)) = \Aut(G)$.
\end{theorem}

In addition, we carried out the same calculation with the subgroups
of
$\sym_{5} \wr \sym_{2}$ of nonzero spread (there are two of them), and
 for both such groups $G$ we found that $\psi(G) = 2$ and there are
no additional automorphisms of 
$\overline{\Gamma}_{\rw}(G)$. That is, both such groups satisfied 
$\Aut(\overline{\Gamma}_{\rw}(G)) = \Aut(G)$.

\end{document}